\RequirePackage{fix-cm}
\documentclass[smallextended]{svjour3}       
\smartqed  
\usepackage{graphicx}
 \usepackage{graphics}
 \usepackage{graphicx}
 \usepackage{epsfig}
 \usepackage[all]{xy}
 \usepackage{cite}

 \usepackage{amssymb}
 \usepackage{amsmath}

 \usepackage{latexsym, epsfig, psfrag,eepic,colordvi,bm}
 \usepackage{graphics,color}
 \usepackage{amsmath,amsfonts,amssymb,amscd}
 \usepackage[all]{xy}
 \usepackage{epstopdf}
 \usepackage{mathrsfs}

 \usepackage{amsmath,amssymb}

 \usepackage{graphicx}
\usepackage{amsmath,amssymb}

\usepackage{graphicx,epsfig}

\usepackage[colorlinks=true,citecolor=black,linkcolor=black,urlcolor=blue]{hyperref}
\usepackage{arydshln}

\begin{document}

\title{Exact enumeration of RNA secondary structures by helices and loops
}

\titlerunning{RNA secondary structures by helices and loops}        

\author{Ricky X. F. Chen$^*$ \and Christian M. Reidys \and Michael S. Waterman
}


\institute{Ricky X. F. Chen$^*$  \at
              School of Mathematics, Hefei University of Technology\\
Hefei, Anhui 230601, P.~R.~China \\
              \email{chenshu731@sina.com}           
           \and
           Christian M. Reidys \at
             $^1$Biocomplexity Institute and Initiative and Dept.~of Math.\\
             University of Virginia, Charlottesville, VA 22904, USA\\
             \email{duckcr@gmail.com}
             \and
            Michael S. Waterman$^{1,2}$ \at
             $^2$Dept. Quantitative and Computational Biology\\
University of Southern California, Los Angeles, CA 90089, USA\\
\email{msw@usc.edu}
}

\date{Received: date / Accepted: date}

\maketitle

\begin{abstract}
Enumerative studies of RNA secondary structures were initiated four decades ago by Waterman and his coworkers.
Since then, RNA secondary structures have been explored according to many different structural characteristics,
for instance, helices, components and loops by Hofacker, Schuster and Stadler, orders by
Nebel, saturated structures by Clote, the $5^{\prime}$-$3^{\prime}$ end distance by Clote, Ponty and Steyaert, and the rainbow spectrum by Li and Reidys. However, the majority of the contributions are asymptotic results, and it is harder to derive explicit formulas.
In this paper, we obtain exact formulas counting RNA secondary structures
with a given number of helices as well as a given joint size distribution of helices and loops,
while some related asymptotic results due to Hofacker, Schuster and Stadler have been known for about twenty years.
Our approach is combinatorial, analyzing a recent bijection between RNA secondary structures and plane trees
discovered by the first author and proposing a variation of Chen's bijective approach of counting trees by forests of simple trees.
\end{abstract}

\section{Introduction}

Ribonucleic acid (RNA) plays an important role in various biological processes
within cells, ranging from catalytic activity to gene expression.
RNA is described by its sequence of bases: A (adenine), U (uracil),
 G (guanine), and C (cytosine). These single-stranded
molecules fold onto themselves forming helical structures, by forming base pairs where A pairs with U while G pairs with C (and sometime
the non-Watson-Crick base pair G with U).  The sequence of bases of
the RNA molecule is known as primary structure, and it is determined experimentally.
A configuration of a molecule with helical structures consistent with a planar graph is known as a secondary structure. RNA sequences have a chemical orientation from their sugar-phosphate backbones. This orientation is designated by end labels $5^{\prime}$ and $3^{\prime}$, with the $5^{\prime}$ end typically appearing on the left.

More than four decades ago,
Waterman and his coworkers pioneered the combinatorics of
RNA secondary structures (Smith and Waterman 1978, Stein and Waterman 1979, Waterman 1978, Waterman 1979). Since then, the combinatorics of RNA secondary structures has been one of the most important topics in computational biology, see Clote (2006), Clote et al.~(2012), Hofacker et al.~(1998),  Heitsch and Poznanovi\'{c} (2013), Han and Reidys (2012), Lorenz et al.~(2008), Li and Reidys (2018), Liao and Wang (2003, 2004), Nebel (2003), and references therein.
In particular, enumeration of RNA secondary structures by some structural characteristics, like hairpins and cloverleaves (Smith and Waterman 1978), and base pairs (Schmitt and Waterman 1994) was first studied. Subsequently, further filtrations of RNA secondary structures were analyzed, as for instance, helices, components, loops by Hofacker, Schuster and Stadler (1998), orders of structures by
Nebel (2003), hairpins and cloverleaves by Liao and Wang (2003), saturated structures by Clote (2006), the $5^{\prime}$-$3^{\prime}$ end distance by Clote, Ponty and Steyaert (2012), Han and Reidys (2012), and the rainbow spectrum by Li and Reidys (2018).
The space of all RNA secondary structures was explored using topology by Penner and Waterman (1993).

It is fair to say that the recursive nature of RNA secondary structures is well understood. Thus, it is possible to use recursion to count secondary structures by certain structural characteristics, and then obtain functional relations of the corresponding generating functions.
However, it is not necessarily easy to derive exact formulae as in Schmitt and Waterman (1994)
and the majority of the contributions in the field is concerned with asymptotic results.
In Schmitt and Waterman (1994), the exact number of secondary structures over a sequence of length $n$ that have $k$ base pairs is obtained by establishing a bijection (see Section~\ref{sec2}) between secondary structures and plane trees,
and then enumerating the corresponding plane trees, e.g.~by Chen's bijective approach (Chen 1990).

In this paper, we focus on helices and loops of RNA secondary structures.
Helices and loops are important characteristics of RNA secondary structures
that have already been extensively studied in many works, as their numbers
and sizes are relevant for computing the free energy of secondary structures and
other purposes (e.g. Zuker, Mathews and Turner 1999).
In particular, for the first time, we obtain exact formulae counting RNA secondary structures with a given number of helices as well as a given joint length distribution of helices and loops.
The asymptotics on the average number of helices (stacks) and the probability for a helix (stack) to have a given length were obtained by Hofacker, Schuster and Stadler (1998) two decades ago.
In addition, we introduce partial stacks as a new characteristic and enumerate secondary structures according to the number and sizes of partial stacks, which
refines the result of Schmitt and Waterman (1994).

Our approach is combinatorial using the new bijection (see Section~\ref{sec2}) between secondary structures and plane trees
discovered by the first author (Chen 2019) and our variation of Chen's bijective approach (Chen 1990) of counting trees by forests of simple trees (see Section~\ref{sec3}).
Through the new bijection in Chen (2019), the key is to realize that the helices of secondary structures are encoded in
the structures of the children of even-level vertices while the loops are encoded in the structures of the children of odd-level vertices of the corresponding plane trees.

\section{Bijections between secondary structures and trees}\label{sec2}

In this section, we review two bijections between RNA secondary structures and plane trees, the Schmitt-Waterman bijection (Schmitt and Waterman 1994)
and the new bijection in Chen (2019).
We first recall the definition of RNA secondary structures following Waterman (1978). Let $[n] = \{ 1,2, \ldots, n\}$. An \emph{RNA secondary
structure} of length $n$ is a simple graph with vertices in $[n]$ and edges in $Y$ satisfying
\begin{itemize}
\item if $(i,j)\in Y$, then $|i-j|\geq 2$;
\item if $(i,j)\in Y$ and $(k,l)\in Y$,
where $i < j$ and $k < l$, and $[i,j] \bigcap [k, l]\neq \varnothing$, then either
$[i,j]\subset [k, l]$ or $[k, l] \subset [i,j]$ (where $[i,j]$ denotes the interval $\{r: i\leq r \leq j\}$).
\end{itemize}
We typically draw an RNA secondary structure in the following manner: we place all vertices in a horizontal line
and we draw an edge as an arc in the upper half-plane. Then, the second condition in the above definition guarantees that any two arcs do not cross with each other.
An arc determines a \emph{base pair}.
The vertex of an arc with a smaller label is called the \emph{left-end} of the arc, and a vertex not adjacent to any edge is called an \emph{isolated base}. In addition, if $(i,j)$ is an arc, we say that another arc $(i_1,j_1)$ (resp. an isolated base $k$) is covered by $(i,j)$ if $[i_1,j_1]\subset [i,j]$ (resp. $k\in [i,j]$), and we also say that the arcs $(i,j)$ and $(i_1, j_1)$ nest with each other.

A \emph{plane tree} $T$ can be recursively defined as an unlabeled tree with one distinguished vertex called the \emph{root} of $T$, where the unlabeled trees obtained by deleting the root as well as its adjacent edges from $T$ are linearly ordered, and they are plane trees with the vertices adjacent to the root of $T$ in $T$ as their respective roots.  In a plane tree $T$, the number of edges in the unique path from a vertex $v$ to the root of $T$ is called the \emph{level} of $v$, and the vertices adjacent to $v$ on a lower level are called the \emph{children} of $v$. The vertices on level $2i$ (resp.~$2i-1$) for $i\geq 0$ are called even-level (resp.~odd-level) vertices. A vertex without any child is called a \emph{leaf}, and an \emph{internal vertex} otherwise. The \emph{outdegree} of a vertex is the number of children of the vertex, i.e.,~one less than the degree of the vertex.
We will draw plane trees with the root on the top level, i.e., level $0$, and with the children of a level $i$ vertex arranged on level $i+1$ left-to-right following their linear order.

{\bf The Schmitt-Waterman bijection (Schmitt and Waterman 1994).} For a given RNA secondary structure, add an arc (i.e.~$(0,n+1)$) covering all bases and view each arc and isolated base as a vertex in a tree rooted at the vertex corresponding to the arc $(0,n+1)$, where the left-to-right children of a vertex $v$ in the tree are the vertices corresponding to the left-to-right arcs and isolated bases directly covered by $v$ (if $v$ is an arc). Therefore,
the Schmitt-Waterman bijection maps an RNA secondary structure with $k$ isolated bases to a plane tree with $k$ leaves.

{\bf Chen's bijection (Chen 2019).} Let $R$ be an RNA secondary structure of length $2a+k$ with $k$ isolated bases. We construct a plane tree $\varphi(R)$ as follows:
\begin{itemize}
	\item[S1:] Add an arc $(0, 2a+k+1)$. Label the isolated bases in $R$ with $b_1,b_2,\ldots, b_k$ left-to-right,
	and label the arcs with $e_0, e_1,\ldots, e_a$ based on the left-to-right order of their left-ends, $e_0$ being $(0, 2a+k+1)$.
	\item[S2:] Let $b_1$ be the root of $\varphi(R)$, and generate $k_1$ children for $b_1$ if there are $k_1$ arcs covering the isolated base $b_1$, where the children from left to right correspond to these $k_1$ arcs from the outermost to the innermost and are labeled correspondingly.
	\item[S3:] For $j=2$ to $k$, put a new child to the left of all existing children of the vertex that corresponds to the innermost arc covering the isolated base $b_j$ in the current partially constructed tree and label the newly generated child with $b_j$. Next generate $k_j$ children for the vertex $b_j$ if there are $k_j$ unused arcs (i.e., those with labels not appearing in the current partial tree) covering the isolated base $b_j$, where again the children from left to right correspond to these $k_j$ arcs from the outermost to the innermost and are labeled accordingly.
\end{itemize}
Then we have: (i) the vertices $b_i$ for all $i$ correspond to the even-level vertices, (ii) the sequence $e_0e_1\cdots e_a$ will be obtained if the children of even-level vertices (in the order $b_1b_2\cdots b_k$) are collected left-to-right sequentially; (iii) the sequence $b_1b_2\cdots b_k$ will be obtained if the even-level vertices are searched by depth-first search from right to left. (i) and (ii) are straightforward, and (iii) can be shown by induction. As a result, the labels of the vertices can be easily and uniquely recovered after being removed. Therefore, the obtained structure with labels removed is a plane tree with $a+k$ edges.

We refer to Schmitt and Waterman (1994) and Chen (2019) for details about the bijections, see Figure~\ref{fig1}.
Combining the two bijections leads to a new bijection on plane trees Chen (2019).
The new bijection in Chen (2019) reveals that the joint distribution of horizontal and vertical decompositions
is in a sense the same as the joint distribution of odd- and even-level vertices over plane trees.

	\begin{figure}[!htb]
		\centering
		\includegraphics[width=1.0\textwidth]{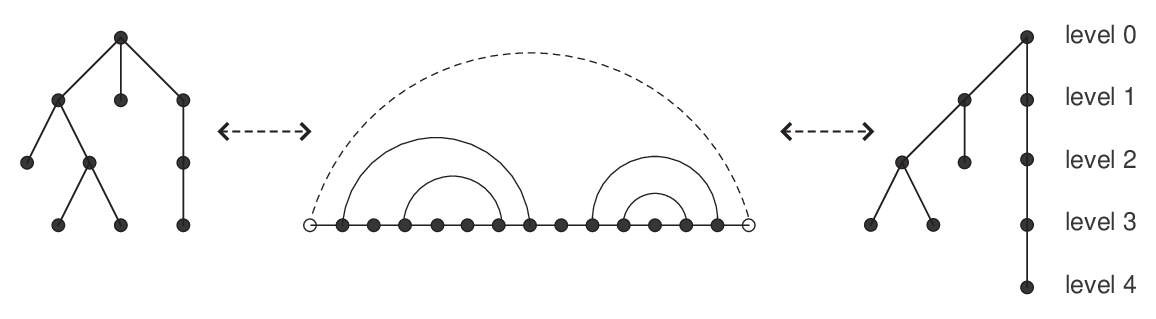}
		\caption{The Schmitt-Waterman bijection (left) and Chen's bijection (right).}
\label{fig1}
	\end{figure}

The results derived here, however, are based on Chen's bijection
and not easy to obtain using the Schmitt-Waterman bijection.

\section{Exact enumeration}\label{sec3}

For simplicity, unless otherwise stated explicitly, we view an RNA secondary structure of length $n$ as the one with the auxiliary arc $(0, n+1)$. So, every RNA secondary structure automatically has at least
one arc.

A \emph{partial stack} is a sequence of consecutive bases $x, x+1, \ldots, x+k-1$ for some $x\geq 0,\; k>0$ that are left-ends of some arcs while $x-1$ and $x+k$
are isolated bases (or not existing). Equivalently, a partial stack is a maximal set of mutually nesting
arcs whose left-ends are consecutive. The number $k$ is called the \emph{length} of the partial stack.
The concept of partial stacks here relates to helices already widely studied in the literature.
A \emph{helix} (or a stack) in an RNA secondary structure is a set of arcs mutually nesting with each other such that
there are no isolated bases between two adjacent arcs (i.e.~one arc and the arc directly covered by it).
That is, both the left-ends and right-ends of a helix are consecutive.

Loops in RNA secondary structures have been extensively studied, as it is important for certain energy models predicting the folded secondary structure
of a given RNA (primary) base sequence. A \emph{loop} consists of a set of isolated bases that are directly covered by the same arc (Hofacker et al.~1998), i.e.~the innermost arc covering these isolated bases.
The \emph{length} of the loop is the size of the set, and
we refer to the arc as \emph{the arc of the loop}.
The \emph{degree} of a loop is one larger than the number of arcs directly covered by the arc of the loop (Hofacker et al.~1998).

Loops have been classified into different types: hairpin loops, interior loops, bulges, and multi-loops. An hairpin loop is a loop where the arc of the loop does not cover any arcs, i.e.~degree one.
An interior loop is a loop where there is exactly one arc directly covered by the arc of the loop, i.e.~degree two.
Bulges are special interior loops.
Other loops, i.e.~degree larger than two, are called multi-loops.

In this paper, for our purpose of enumeration, we define the \emph{size} of a loop to be the number of elements, either isolated bases or arcs, immediately covered by the arc of the loop, that is, one less than the sum of the degree and the length of the loop.
Furthermore, we view every arc as the arc of some loop, thus loops of length zero (i.e.~empty loops)
are allowed.

In the following, we explicitly enumerate RNA secondary structures according to parameters concerning partial stacks, helices and loops.
We first show how Chen's bijection immediately allows us to derive results by applying the Lagrange inversion formula. Next, we obtain the rest of results in a completely combinatorial manner by decomposing certain plane trees into forests of even simpler trees.

\subsection{Combinatorics on loops and partial stacks}
We first enumerate RNA secondary structures with constraints on partial stacks and loops, e.g.~the maximal length and the number.

We recall the following version of the multivariate (bivariate) Lagrange inversion formula (Bender and Richmond 1998, Gessel 1987).
Let $g(x_1,x_2)$, $f_1(x_1,x_2)$, $f_2(x_1,x_2)$ be formal power series in $x_1,~x_2$ such that $f_i(0,0)\neq 0$. Then the set of equations $w_i=t_if_i(w_1,w_2)$ for $1\leq i \leq 2$ uniquely determine the $w_i$ as
formal power series in $t_1,~t_2$, and
$$
[t_1^p t_2^q]g(w_1, w_2)=[x_1^p x_2^q] g(x_1,x_2)f_1^p(x_1,x_2)f_2^q(x_1,x_2)	\det
\begin{Bmatrix}
1-\frac{x_1}{f_1}\frac{\partial f_1}{\partial x_1} & -\frac{x_1}{f_2}\frac{\partial f_2}{\partial x_1}\\
-\frac{x_2}{f_1}\frac{\partial f_1}{\partial x_2} & 1-\frac{x_2}{f_2}\frac{\partial f_2}{\partial x_2}
\end{Bmatrix},
$$
where $[t_1^p t_2^q]$ denotes the coefficient of $t_1^p t_2^q$.

The following lemma follows from Chen's bijection.
\begin{lemma}\label{lem:3lem1}
The number of RNA secondary structures with $b+1$ ($b\geq 0$) base pairs, $k$ isolated bases and $l$ partial stacks
is the same as the number of plane trees with $b+k$ edges where there are $l$ even-level internal vertices
out of a total of $k$ even-level vertices.
\end{lemma}

\begin{theorem}\label{thm:3thm1}
The number of RNA secondary structures with $b+1$ base pairs, $k$ isolated bases and $l$ partial stacks
is given by
\begin{align*}
\begin{cases}
\frac{1}{l-1} {l+b-1 \choose l-2} {b+1 \choose l}{b+k-1 \choose k-l}, & \mbox{for $l>1$};\\
{b+k-1 \choose k-1}, & \mbox{for $l=1$}.
\end{cases}
\end{align*}
\end{theorem}

\begin{proof}
Based on Lemma~\ref{lem:3lem1}, we enumerate the corresponding plane trees $T$.
Let $C_l(b,k)$ be the desired number. Then, ${C}_l(b,l)$ is the number of plane trees of $b+l$
edges where there are $b+1$ odd-level vertices and $l$ even-level internal vertices.
We first have
$$
C_l(b,k)={b+l+k-l-1 \choose k-l} {C}_l(b,l)={b+k-1 \choose k-l} {C}_l(b,l).
$$
That is, from each of the ${C}_l(b,l)$ trees without even-level leaves, we can insert $k-l$
leaves at the sectors around the odd-level vertices to obtain plane trees with
$k$ even-level vertices and $l$ even-level internal vertices.
Note that the total number of sectors around the odd-level vertices
is the total degree of the odd-level vertices. The latter equals half of the total degree of all vertices, i.e.~$b+l$, whence
each will generate ${b+k-1 \choose k-l}$ distinct desired plane trees.
Next, we compute ${C}_l(b,l)$.

Given two sets $E$ and $O$ of vertices, we call a plane tree $T$ a set-alternating tree if the vertices on any path starting from the root of $T$ alternate in the two sets. Let $\mathscr{P}_E$ denote the set of set-alternating plane trees with root in $E$ and every $E$-vertex is an internal vertex and let  $\mathscr{P}_O$ denote the set of set-alternating plane trees with root in $O$ and every $E$-vertex if any is internal.
Consider the generating function
$$
w_1(t_1,t_2) =\sum_{T\in \mathscr{P}_E} t_1^{\mbox{\#vertices in $E$ in $T$}}t_2^{\mbox{\#vertices in $O$ in $T$}}.
$$
Clearly, the number $C_l(b,l)$ of plane trees of $b+l$
edges where there are $b+1$ odd-level vertices and $l$ even-level internal vertices is the same as the number of set-alternating trees of $b+l$
edges with root in $E$ such that there are $b+1$ odd-level vertices and $l$ even-level internal vertices. Thus, we have ${C}_l(b,l)=[t_1^l t_2^{b+1}] w_1$.

Note that the subtrees of a tree $T\in \mathscr{P}_E $ are trees in $\mathscr{P}_O$ and there is at least one
such a subtree as we require every $E$-vertex in $T$ is internal.
Let $w_2(t_1, t_2)$ be the generating function for trees in $\mathscr{P}_O$, i.e.,
$$
w_2(t_1,t_2) =\sum_{T\in \mathscr{P}_O} t_1^{\mbox{\#vertices in $E$ in $T$}}t_2^{\mbox{\#vertices in $O$ in $T$}}.
$$
Then, we first have
$$
w_1=t_1(w_2+w_2^2+\cdots)=t_1\frac{w_2}{1-w_2}.
$$
Similarly, the subtrees of a tree $T\in \mathscr{P}_O $ are trees in $\mathscr{P}_E$, but there could be by definition no subtrees at all.
Thus, we have the equation
$$
w_2=t_2(1+w_1+w_1^2+\cdots)=t_2\frac{1}{1-w_1}.
$$
In terms of the above bivariate Lagrange inversion formula, we obtain
$$
g(x_1,x_2)=x_1, \quad f_1(x_1,x_2)=\frac{x_2}{1-x_2}, \quad f_2(x_1,x_2)=\frac{1}{1-x_1}.
$$
Compute
\begin{align*}
	[t_1^pt_2^q]w_1&=[x_1^px_2^q] g\cdot f_1^p\cdot f_2^q \cdot
	\det
	\begin{Bmatrix}
		1-\frac{x_1}{f_1}\frac{\partial f_1}{\partial x_1} & -\frac{x_1}{f_2}\frac{\partial f_2}{\partial x_1}\\
		-\frac{x_2}{f_1}\frac{\partial f_1}{\partial x_2} & 1-\frac{x_2}{f_2}\frac{\partial f_2}{\partial x_2}
	\end{Bmatrix}\\
	&= [x_1^px_2^q] \frac{x_2^p}{(1-x_2)^p}\frac{x_1}{(1-x_1)^q}\left(1-\frac{x_1(1-x_2)}{(1-x_1)}\frac{1}{(1-x_2)^2} \right)\\
	&=[x_1^{p-1} x_2^{q-p}]\bigl( (1-x_1)^{-q}(1-x_2)^{-p}-x_1(1-x_1)^{-q-1} (1-x_2)^{-p-1} \bigr)\\
	&={q+p-2 \choose p-1}{q-1 \choose q-p}- {q+p-2 \choose p-2}{q \choose q-p}\\
	&=\frac{1}{p-1} {p+q-2 \choose p-2} {q \choose q-p}.
\end{align*}
Note that the last formula only holds for $p>1$. In case of $p=1$, the second last formula gives us one.
Setting $p=l$ and $q=b+1$ completes the proof.
 \end{proof}

We next analyze partial stacks and first note:
\begin{lemma}\label{lem:3lem3}
Let $R$ be an RNA secondary structure of length $n$.
The length of a partial stack uniquely corresponds to the outdegree of an even-level vertex in $\varphi(R)$.
\end{lemma}

Now we are in a position to compute the number of RNA secondary structures
with every partial stack of length at most $h$.

\begin{theorem}\label{thm:max-h}
	The number $H_h(b,k)$ of RNA secondary structures $R$ with $b+1$ base pairs and $k$ isolated bases such that every partial stack has a length at most $h$ is given by
		\begin{align}
		H_h(b,k)=
			\frac{1}{b+k}{b+k \choose k}\sum_{i\geq 0}^{i\leq \frac{b+1}{h+1}}(-1)^i{k\choose  i}{b+k-i(h+1) \choose k-1}.	
		\end{align}		
\end{theorem}
\begin{proof}
Based on Lemma~\ref{lem:3lem3}, the number $H_h(b,k)$ equals the number of plane trees $T$ with $b+k$ edges where there are $k$ even-level vertices
and every even-level vertex has at most $h$ children. The latter can be computed as shown below:

Given two sets $E$ and $O$ of vertices, let
\begin{align*}
\zeta_1(t_1,t_2) &=\sum_{T\in \mathscr{H}_E} t_1^{\mbox{\#vertices in $E$ in $T$}}t_2^{\mbox{\#vertices in $O$ in $T$}},\\
\zeta_2(t_1,t_2) &=\sum_{T\in \mathscr{H}_O} t_1^{\mbox{\#vertices in $E$ in $T$}}t_2^{\mbox{\#vertices in $O$ in $T$}},
\end{align*}
where $\mathscr{H}_E$ denotes the set of set-alternating plane trees with root in $E$ and every $E$-vertex having at most $h$ children while  $\mathscr{H}_O$ denotes the set of set-alternating plane trees with root in $O$ and every $E$-vertex having at most $h$ children. Then,
$$
\zeta_1=t_1\frac{1-\zeta_2^{h+1}}{1-\zeta_2}, \quad \zeta_2=t_2\frac{1}{1-\zeta_1}.
$$
Clearly, we have
$$
H_h(b,k)=[t_1^k t_2^{b+1}] \zeta_1.
$$
Analogously, based on the above multivariate Lagrange inversion formula, we obtain the formula for $H_h(b,k)$.
 \end{proof}

We remark that Theorem~\ref{thm:max-h} is essentially the same as Theorem~$3.4$ regarding enumeration of certain class of trees in Chen (2019) and more details can be found there.
Note that for any RNA secondary structure with $b+1$ base pairs, a partial stack can have a length at most $b+1$.
Thus, $H_{b+1}(b,k)$ counts the total number of RNA secondary structures with $b+1$ base pairs and $k$ isolated bases. From this we can recover:

\begin{corollary}[Schmitt-Waterman 1994]\label{cor:narayana}
The number of RNA secondary structures with $b+1$ base pairs and $k$ isolated bases is given by the Narayana number
$$
\frac{1}{b+k}{b+k\choose k} {b+k\choose k-1}=H_{b+1}(b,k).
$$

\end{corollary}

Based on Theorem~\ref{thm:3thm1} and Corollary~\ref{cor:narayana}, we immediately have
the following expression of the Narayana number in the form of a sum which seems new.
\begin{corollary} The following identity holds:
\begin{align*}
\frac{1}{b+k}{b+k\choose k} {b+k\choose k-1}={b+k-1 \choose k-1}+\sum_{ l >1} \frac{1}{l-1} {l+b-1 \choose l-2} {b+1 \choose l}{b+k-1 \choose k-l}.
\end{align*}
\end{corollary}

Applying Theorem~\ref{thm:3thm1} and Corollary~\ref{cor:narayana}, it is easy to compute the total number of partial stacks and the total number of RNA secondary structures, respectively.
As a result, we can obtain the exact expected number of partial stacks.
Along these lines, we can obtain the expected length of a partial
stack in a random RNA secondary structure.

\begin{lemma}\label{lem:loop-length}
Let $R$ be an RNA secondary structure of length $n$.
The size of a loop uniquely corresponds to one plus the outdegree of an odd-level vertex in $\varphi(R)$.
\end{lemma}
\begin{proof}
Note that, by definition, each arc of an RNA secondary structure covers (not necessarily directly) at least
one isolated base. In Chen's bijection, an arc $e$ must be a child of the leftmost isolated base covered by $e$
in $\varphi(R)$. With the exception of the leftmost element (either an arc covering the leftmost isolated base or the leftmost isolated base itself) directly covered by $e$, any other element directly covered by $e$ induces a unique child of $e$ in $\varphi(R)$.
That is, an isolated base other than the leftmost one directly covered by $e$ clearly induces a child (a leaf) of $e$;
for an arc $ e'$ directly covered by $e$, the leftmost isolated base covered by $e'$ must induce a child
of $e$ in $\varphi(R)$ and everything else covered by $e'$ will be in the subtree rooted on the induced child.
Therefore, the number of children of $e$ in $\varphi(R)$ will be one less than the number of elements
directly covered by $e$ in $R$ (i.e.~the size of the loop having $e$ as its arc), and the lemma follows.
\end{proof}

\begin{theorem}\label{thm:max-l}
	The number $L_l(b,k)$ of RNA secondary structures $R$ with $b+1$ base pairs and $k$ isolated bases such that every loop has a size at most $l$ is given by
		\begin{align}
		L_l(b,k)=	\frac{1}{b+1}{b+k \choose b} \sum_{i\geq 0}^{i\leq \frac{k-1}{l}} (-1)^i {b+1\choose i} {b+k-1-il \choose k-1-il} \; .
 		\end{align}		
\end{theorem}
\begin{proof}
Based on Lemma~\ref{lem:loop-length}, the number $L_l(b,k)$ equals the number of plane trees $T$ with $b+k$ edges where there are $b+1$ odd-level vertices
and every odd-level vertex has at most $l-1$ children. The latter can be computed as shown below:

Given two sets $E$ and $O$ of vertices, let
\begin{align*}
\xi_1(t_1,t_2) &=\sum_{T\in \mathscr{L}_E} t_1^{\mbox{\#vertices in $E$ in $T$}}t_2^{\mbox{\#vertices in $O$ in $T$}},\\
\xi_2(t_1,t_2) &=\sum_{T\in \mathscr{L}_O} t_1^{\mbox{\#vertices in $E$ in $T$}}t_2^{\mbox{\#vertices in $O$ in $T$}},
\end{align*}
where $\mathscr{L}_E$ denotes the set of set-alternating plane trees with root in $E$ and every $O$-vertex having at most $l-1$ children while  $\mathscr{L}_O$ denotes the set of set-alternating plane trees with root in $O$ and every $O$-vertex having at most $l-1$ children. Then, we have
$$
\xi_2=t_2\frac{1-\xi_1^{l}}{1-\xi_1}, \quad \xi_1=t_1\frac{1}{1-\xi_2}.
$$
Clearly,
$$
L_l(b,k)=[t_1^k t_2^{b+1}] \xi_1
$$
and employing the multivariate Lagrange inversion formula, we obtain the formula for $L_l(b,k)$.
\end{proof}

For $l$ being sufficiently large, we recover Corollary~\ref{cor:narayana}.
Furthermore,

\begin{theorem}\label{thm:max-hl}
	The number $P_{h,l}(b,k)$ of RNA secondary structures $R$ with $b+1$ base pairs and $k$ isolated bases such that every partial stack has a length at most $h$ and every loop has a length at most $l$ is given by
		\begin{multline}
	P_{h,l}(b,k)=[x_1^{k-1}x_2^{b+1}] \frac{(1-x_2^{h+1})^k}{(1-x_2)^k}\frac{ (1-x_1^{l})^{b+1}}{(1-x_1)^{b+1}}\\
	\times \left(1-\frac{x_1x_2(1-x_2)(1-x_1)}{(1-x_1^l)(1-x_2^{h+1})}\Big[\frac{1-x_2^h}{(1-x_2)^2}-\frac{hx_2^h}{1-x_2}\Big] \Big[\frac{1-x_1^{l-1}}{(1-x_1)^2}-\frac{(l-1)x_1^{l-1}}{1-x_1}\Big]
	\right).	
		\end{multline}		
\end{theorem}

\subsection{Joint distribution of helices and loops}
Next we enumerate RNA secondary structures with a given joint size distribution
of helices and loops.

In a plane tree, we call a maximal set of (left-to-right) consecutive odd-level leaves
and an immediately following odd-level internal vertex if any
that are children of the same even-level vertex a \emph{E-block}.
By definition, for an even-level internal vertex $v$, if its rightmost child
is an internal vertex, then the number of E-blocks of $v$ (i.e.~determined by
$v$) is the number of odd-level internal vertices that are children of $v$;
otherwise, the number of E-blocks determined by $v$ is one plus the number of odd-level internal vertices that are children of $v$. Accordingly, we have the following lemma:

\begin{lemma}\label{lem:num-eblock}
The number of E-blocks in a plane tree equals the sum of the number of odd-level internal vertices and
the number of even-level vertices whose rightmost children are leaves.
\end{lemma}

Inspecting Chen's bijection $\varphi$, we realize that E-blocks encode the information of helices in the secondary structures.
\begin{proposition}\label{prop:eblock}
Let $R$ be an RNA secondary structure.
Then, the number of helices in $R$ is the number of E-blocks in $\varphi(R)$.
In addition, the size of a helix of $R$ uniquely
corresponds to the size of an E-block in $\varphi(R)$.
\end{proposition}
\begin{proof}
We have shown that the children of an even-level internal vertex corresponds to a
partial stack. Since a partial stack is a set of mutually nesting
arcs whose left-ends are consecutive, a helix is contained in a partial stack and a partial stack may consist of
multiple helices.
Thus, it is necessary to look at the right-ends of a partial stack,
since any maximal set of consecutive right-ends determines a helix.

If the right-ends of one arc $e$ and the arc immediately covered by $e$ in the same partial stack are
not consecutive, there must be at least one isolated base between the two right-ends.
According to Chen's bijection, the leftmost such an isolated base must induce a child of $e$
in $\varphi(R)$, so the arc $e$ is an odd-level internal vertex,
see Figure~\ref{fig:stack-e-block}.
\begin{figure}[!htb]
		\centering
		\includegraphics[width=0.9\textwidth]{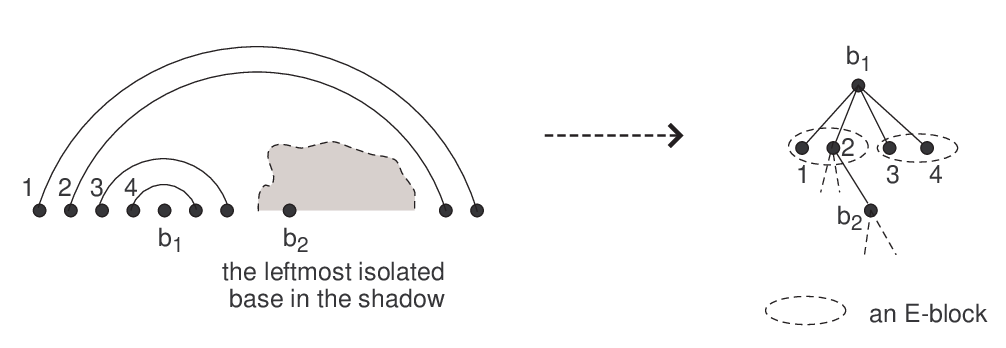}
		\caption{From helices to E-blocks}
\label{fig:stack-e-block}
	\end{figure}
Note that in a partial stack the arcs from the outermost to the innermost are arranged
from left to right as children of some even-level internal vertex.
Therefore, each odd-level internal vertex and the leaves to its left, i.e.~an E-block, determines a helix,
except for that the rightmost E-block (i.e.~the helix determined by the inner most arc) may not contain
an odd-level internal vertex (for instance, in Figure~\ref{fig:stack-e-block} vertex $4$ is not an internal vertex),
and the proposition follows.
\end{proof}

\emph{Remark.} Based on the Schmitt-Waterman bijection, it is easy to observe that helices (and partial stacks) correspond to certain
special paths in the corresponding plane trees. However, paths, in a sense, do not reflect the obvious recursive organization of plane trees
as E-blocks (or degrees) do, which makes enumeration harder.

Accordingly, counting secondary structures by helices can be done by counting (set-alternating) plane trees according to E-blocks.
The latter can be achieved by enumerating labelled set-alternating trees according to E-blocks,
as the labelled case and the unlabelled case just differ by a multiplicity.
A \emph{labelled set-alternating E-tree (resp.~O-tree)} is a plane tree where the even-level vertices carry distinguishable labels from a set $E$ (resp.~$O$)
and the odd-level vertices carry distinguishable labels from a set $O$ (resp.~$E$).
We shall enumerate labelled set-alternating E-trees based on our variation
of Chen's bijection (Chen 1990) dealing with uniformly (i.e.~using only one set of labels) labelled plane trees.

In order to state our bijection,
we set some notation.
A \emph{forest} is a set of trees.
An even-level internal vertex whose children are all leaves is
called \emph{young}.
A (labelled) \emph{small tree} is a (labelled) tree with only two levels in total.
A \emph{small set-alternating tree} is a set-alternating tree with
only two levels.
The \emph{size} of a small tree is the number of edges in the small tree.
We shall call a small set-alternating tree
with a root in $E$ (resp.~$O$) a small \emph{E-tree} (resp.~\emph{O-tree}).

Our bijection is between certain labelled set-alternating trees and certain forests consisting of some small E-trees and some small O-trees, where additional labels besides the ones used in the plane trees will appear in the forests. We shall mark each of the additional labels by $^*$ and refer to them as starred labels.
In contrast to Chen's bijection (Chen 1990), we need to deal with two sets of labels and carefully control their distribution
in the forests.
As the local structure of the children (e.g.~distribution of leaves and internal vertices) of an odd-level internal vertex is not relevant, we call all children of an odd-level internal vertex an O-block.
In the following, we sometimes refer to a vertex by its label.

\begin{theorem}\label{thm:bij2}
Suppose $E=[k]\subset E^*=[k]\bigcup \{(k+1)^*, \ldots, (k+s-1)^*\}$ and $O=[\overline{b+1} ] \subset O^*=[\overline{b+1}] \bigcup \{(\overline{b+2})^*, \ldots, (\overline{b+1+l_o})^*\}$. Then there is a bijection between the set $\mathbb{T}$ of set-alternating E-trees over $E \bigcup O$ with $b+k$ edges, in which
\begin{itemize}
\item[(a)] there are $s$ E-blocks, and
\item[(b)] there are $l_e$ even-level internal vertices, $y$ of which being young, and
\item[(c)] there are $l_o$
odd-level internal vertices ($b+1 \geq s\geq l_e \geq y$ and $s \geq l_o$),
\end{itemize}
and
the set of forests of small set-alternating trees over $E^* \bigcup O^*$, $\mathbb{F}$, in which:
\begin{itemize}
\item[(a*)] there are in total $s$ small E-trees and $l_e$ of them have roots from $E$, while $s-l_e$ of them have roots from $E^*\setminus E$, and
\item[(b*)] there are $l_o$ small O-trees, each of which having a root from $O$, and there are $l_o$ small E-trees having a unique starred leaf which is the rightmost leaf, and
\item[(c*)] $y$ small E-trees have no starred labels, and the smallest $y$ labels in $E^*\setminus E$ are leaves of some small O-trees, and
\item[(d*)] there are $s-y-l_o$ small E-trees which have starred roots and unstarred leaves. Furthermore, if $(k+m)^*$ is the root of such E-tree
and $m\neq s-1$, then $(k+m+1)^*$ is a leaf in some small O-tree.
\end{itemize}
\end{theorem}
\begin{proof}
We first introduce a linear order `$<$' on $E\bigcup O$ such that $\overline{j}<\overline{j+1}$ and any element $i\in [k]$ is smaller than $\overline{j}\in [\overline{b+1}]$, i.e.~$i< \overline{j}$.
We also refer to the elements in $E^*$ (resp.~$O^*$) as E-labels (resp.~O-labels).

{\bf The map $h: \mathbb{T} \rightarrow \mathbb{F}$.}
For each $T \in \mathbb{T}$, we successively construct a forest $F \in \mathbb{F}$ of small trees according to the following iterative procedure $h$.
\begin{itemize}
\item[i.] Let $i_e,~i_o$ be non-negative integers initializing $i_e=0,\, i_o=0$ and initializing $F=\varnothing$. If $T$ is not a small E-tree carrying at most one starred leaf, execute step ii and step iii. Otherwise, add $T$ to $F$.
\item[ii.]
In $T$, find the minimal internal vertex whose the leftmost block (either E-block or O-block) of vertices are leaves. Here the blocks, in particular the E-blocks, still refer to the blocks (of vertices) inherited from the initial tree instead of those determined by definition in the current tree.
Remove the small tree determined by $v$ and the leftmost block from $T$, together with all labels, and add the small tree to $F$.
\item[iii.] If $v$ has a label in $[k]$, relabel $v$ with the label $(d+i_e+1)^*$ in the remaining tree but use the original unstarred label of $v$ for comparison in any further executions of step ii, and update $T$ as
the resulting tree, and set $i_e=i_e+1$. If $v$ has a label in $[\overline{b+1}]$, place a vertex with label $(\overline{t+i_o+1})^*$ in the remaining tree, and update $T$ as
the resulting tree, and set $i_o=i_o+1$.
\end{itemize}
It should be obvious that the tree $T$ always has a root in $E^*$, i.e.~$T$ is a E-tree, and the minimal internal vertex sought in step ii always exists.
The last small tree joining $F$, or equivalently the last updated $T \neq \varnothing$, is associated to the rightmost E-block of the root of the initial tree, and depending on if the rightmost child of the root is internal, the last small tree may have a starred leaf.
In addition, if the last small tree has a starred E-label (in the case of the root having more than one E-block), then the starred E-label must be $(k+s-1)^*$, i.e.~the last starred E-label.
In the end, we obtain $s$ small E-trees and $l_o$ small O-trees
as each block in the initial tree is associated to a unique small tree.
An illustration of the process is in Figure~\ref{fig-block}.
\begin{figure}[!htb]
		\centering
		\includegraphics[width=0.9\textwidth]{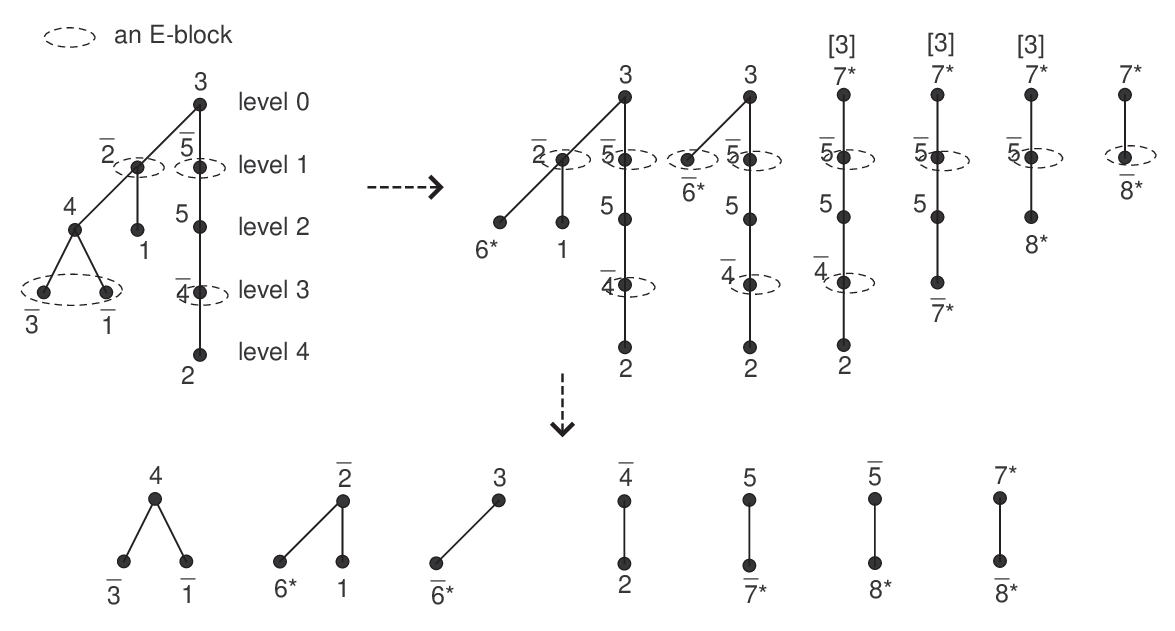}
		\caption{An example from a tree to a forest of small trees.}
\label{fig-block}
	\end{figure}
	
Note that there are four types of small trees: unstarred root and unstarred leaves, unstarred root and starred leaves,
starred root and starred leaves, and starred root and unstarred leaves.
In order to characterize these small trees, we need to know the quantity of the starred labels and unstarred labels
and how they distribute among the roots and leaves of the small trees.
We observe
\begin{itemize}
\item Each O-block induces a starred O-label so there are $l_o$ starred O-labels.
Each E-block, except the last removed one, induces a starred E-label so there are $s-1$ starred E-labels generated.
Thus the labels in the obtained forest constitute $E^*\bigcup O^*$.
\item
Note that each even-level internal vertex contributes exactly one small E-tree whose root has an
unstarred label.
Thus, there are $l_e$ small E-trees having roots from $E$, and (a*) holds.
\item Each odd-level internal vertex contributes exactly one small O-tree whose root has an
unstarred label.
Thus, there are $l_o$ small E-trees having roots from $O$.
Note that each E-block contains at most one odd-level internal vertex which implies
that each small E-tree contains at most one starred O-label.
As a result, there are exactly $l_o$ small E-trees whose rightmost leaf has a starred O-label and (b*) follows.
\item By construction, each young even-level internal vertex has only one E-block so each
young even-level internal vertex generates exactly one small E-tree containing no starred labels.
Note that these $y$ young internal vertices are removed before other even-level internal vertices.
Thus, the first (smallest) $y$ corresponding starred E-labels induced by the young internal vertices
are leaves in some small O-trees.
Therefore, we have (c*).
\item Since each small E-tree can contain at most one starred O-label,
there are $s-l_o-y$ small E-trees that have starred roots and unstarred leaves.
Any such small E-tree stems from the rightmost E-block $B$ of an even-level internal vertex, which is not young and whose rightmost child is a leaf, otherwise the rightmost leaf must have a starred O-label and/or the root has an unstarred label. In addition, if the starred root is $(k+m)^*$ with $m\neq s-1$,
the even-level internal vertex can not be the root of the initial tree, as discussed earlier.
Suppose the even-level internal vertex has an unstarred label $b$ and $B$ is its rightmost E-block.
In the procedure $h$, at some point, the E-block $C$ immediately to the left of $B$ is removed,
and at the same time, the vertex $b$ obtains a starred E-label $(k+m)^*$.
Then $b$ has only one E-block remaining.
We claim that the very next removed block must be the block $B$ hence we obtain a small E-tree
with root $(k+m)^*$, and at the same time create the starred label $(k+m+1)^*$ as a leaf of some odd-level internal vertex. Since just before the removal of $C$, the unstarred label $b$ is the minimal among those whose leftmost block of vertices are leaves. After the removal of $C$, this is still the case, as by construction, the vertices in $B$ are all leaves,
see Figure~\ref{fig:prop-d}.
As a result, (d*) holds.
\begin{figure}[!htb]
		\centering
		\includegraphics[width=0.9\textwidth]{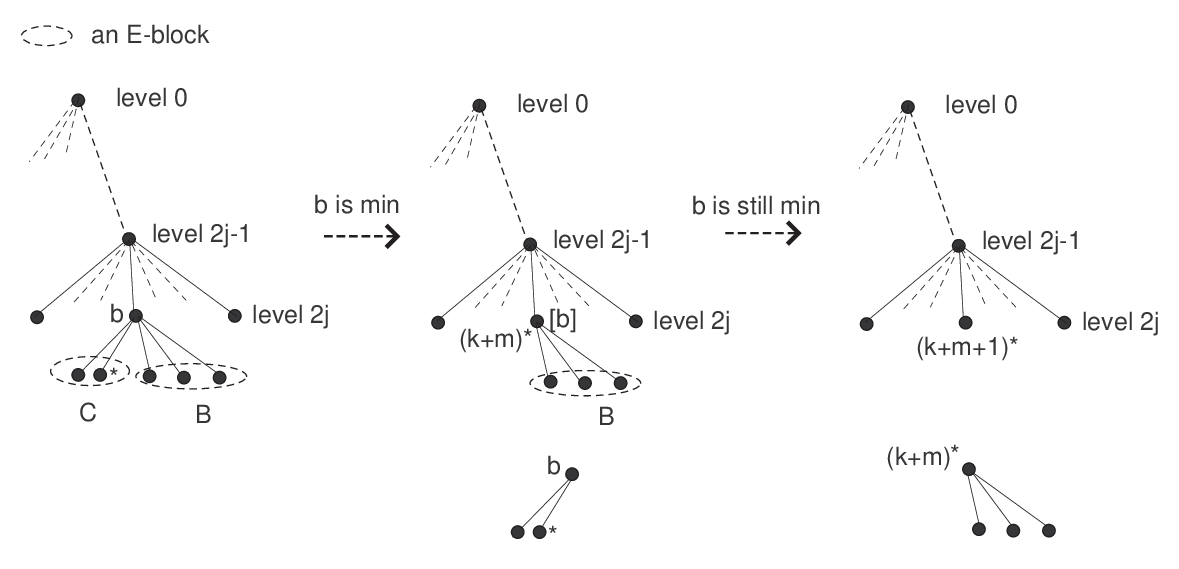}
		\caption{An illustration of property (d*).}
\label{fig:prop-d}
	\end{figure}
\end{itemize}

{\bf The map $g:  \mathbb{F}\rightarrow \mathbb{T}$.}
Next, we specify the procedure $g$ that maps a forest $F\in \mathbb{F}$ to a tree $T \in \mathbb{T}$. \begin{itemize}
\item[(i)] Find a tree in $F$  with the minimum root such that there is no vertex with a starred label in the tree.
If the root $v$ of the found tree is an $E$-vertex,  then merge the root
with the vertex having the minimum label $u^*$ in the set $\{ (k+1)^*, (k+2)^*,\ldots, (k+s-1)^*\}$ in $F$, and label the newly generated vertex by the label of $v$ (and discard $u^*$). In addition, if $u^*$ is also a root, then put
the found tree (with root $v$) on the left-hand side when merging. If the root $v$ of the found tree is an $O$-vertex, then merge the root
with the vertex having the minimum label in the set $\{ (\overline{b+2})^*, (\overline{b+3})^*,\ldots, (\overline{b+1+l_o})^*\}$, and label the newly generated vertex by the label of $v$. Update $F$ as the resulting forest of trees.
Here we remark that a tree without any starred label always exists, because there are in total $s+l_o$ trees and $s-1+l_o$
starred labels at the beginning, and later every time we decrease the number of
starred E-labels (resp.~O-labels) by one, we decrease the number of E-trees (resp.~O-trees) by one at the same time.
In addition, since the number of E-trees is always one larger
than the number of starred E-labels, we eventually obtain an E-tree.
\item[(ii)] Iterate (i) until $F$ becomes a single labelled plane tree $T$.
\end{itemize}
Note that there are two types of mergings, either vertically attaching a tree without
starred labels to a starred leaf of another tree, or horizontally merging the root of a tree without starred labels and
a starred root of another tree.
The numbers of even-level internal vertices and odd-level internal vertices in $T$ are determined by
the number of small trees with unstarred roots and are obviously $l_e$ and $l_o$.
It remains to check young even-level internal vertices and E-blocks.

\emph{Young internal vertices:}
According to the above procedure, we have to process these $y$ small E-trees without any starred labels first.
By assumption, the $y$ smallest starred labels in $E^* \setminus E$ are leaves in some O-trees.
Then, these $y$ small E-trees will be attached to these $y$ starred leaves sequentially.
Another fact validating this statement is that each merging of these $y$ mergings does not generate
an E-tree with an even smaller root than the ones of the remaining small E-trees.
As a consequence, each of these $y$ small E-trees contributes a young even-level internal vertex,
and this is the only way young internal vertices arise.
Thus, there are in total $y$ young even-level internal vertices.

\emph{E-blocks:} We first claim
any small E-tree with a starred root and unstarred leaves
will appear as the rightmost block of some even-level internal vertex.

Let $T_e$ be a small E-tree with a starred root $(s+m)^*$ and unstarred leaves.
Suppose in the procedure $g$, $(s+m)^*$ is merged with the root of some tree $T_b$ without any starred labels, with $T_e$ on the right hand side. Suppose the root of $T_b$ has unstarred label $b$. It is necessary that
$b$ is minimal right before the merge.
Note that the minimal tree without any starred labels that we can find next still has root $b$,
because the merge just completed did not generate any smaller (in terms of the labels of the roots) trees.
Next, if $m=s-1$, there will be no merges involving E-labels anymore, since we have used up all starred E-labels, the E-block $T_e$ is accordingly the rightmost one.
If $m\neq s-1$, then the next found minimum starred E-label must be $(k+m+1)^*$.
Since the latter is a leaf of some odd-level vertex ($\overline{j}$), the tree with root $b$ will be attached to it.
As a result,
the E-block involved in the last step remains the rightmost block and will not be impacted later,
whence the claim,
see Figure~\ref{fig:bad-good}.

\begin{figure}[!htb]
		\centering
		\includegraphics[width=0.9\textwidth]{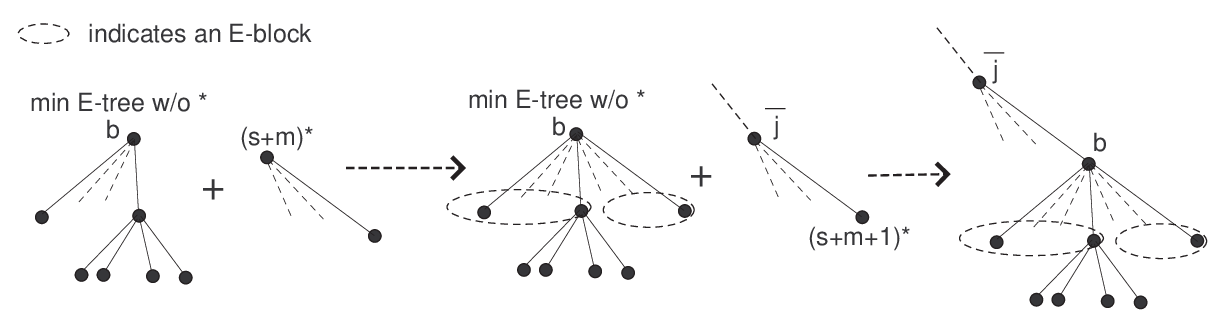}
		\caption{Appearing as the rightmost E-block.}
\label{fig:bad-good}
	\end{figure}

Accordingly, in the resulting tree $g(F)$, there are $y+(s-y-l_o)$ even-level internal vertices
whose rightmost children are leaves. Based on Lemma~\ref{lem:num-eblock},
we conclude that the number of E-blocks in $g(F)$ is $y+(s-y-l_o)+l_o=s$.

{Note that in $g$ only the relative order of the starred labels matters.}

We next show that $h$ and $g$ satisfy $g (h(T))=T$. Given $T$ and the corresponding forest $F=h(T)$, suppose the resulting tree after removing the first small tree (determined by a vertex $v$) is $T'$ and the resulting forest with the first small tree removed from $F$ is $F'$.
Without loss of generality, we assume $v$ is an even-level internal vertex and initially has the unstarred label $1$ and has only one block.
Let $T''$ and $F''$ be obtained by replacing the induced (by $1$) starred label $(k+1)^*$ in $T'$ and $F'$ with $1$, respectively. Suppose the blocks of $T''$ are induced by $T$ instead of being determined by definition in $T''$ itself. By construction, $T$ and $T''$ only differ at vertex $1$.

\emph{Claim~$1$.} Let $F'''$ be obtained from $F''$ by shifting each starred E-label (O-label in case of $v$ being an odd-level internal vertex) if any down by one. Then, $h(T'')=F'''$.

\emph{ Induction base.} Obviously, if $T$ has only one block, then $F$ has a unique small tree, and clearly $g (h(T))=T$.

\emph{ Induction step.} Suppose $g(F''')=T''$. As already remarked, this implies $g(F'')=T''$.
Our objective is to show that $g(F)=T$.
Let $\overline{F''}$ be the forest after the first merger of $g$ from $F$. It is obvious that $F''$ and $\overline{F''}$ differ only at the trees carrying vertex $1$ exactly as $T$ and $T''$ differ at vertex $1$.

\emph{Claim~$2$.} After the same number of mergers from $\overline{F''}$ and $F''$, the respective resulting forests still only differ at vertex $1$ as described above.\\
This follows from the fact that subsequent mergers never involve vertex $1$.
Note that by induction hypothesis $F''$ eventually arrives at $T''$.
We knew $T$ is the object that differs with $T''$ only at vertex $1$,
thus $\overline{F''}$ must eventually arrives at $T$,
see Figure~\ref{fig:show-bij}.
Therefore, $g(F)=g(h(T))=T$,
 completing the induction proof,
 and the theorem follows.
 \end{proof}
\begin{figure}[!htb]
		\centering
		\includegraphics[width=0.85\textwidth]{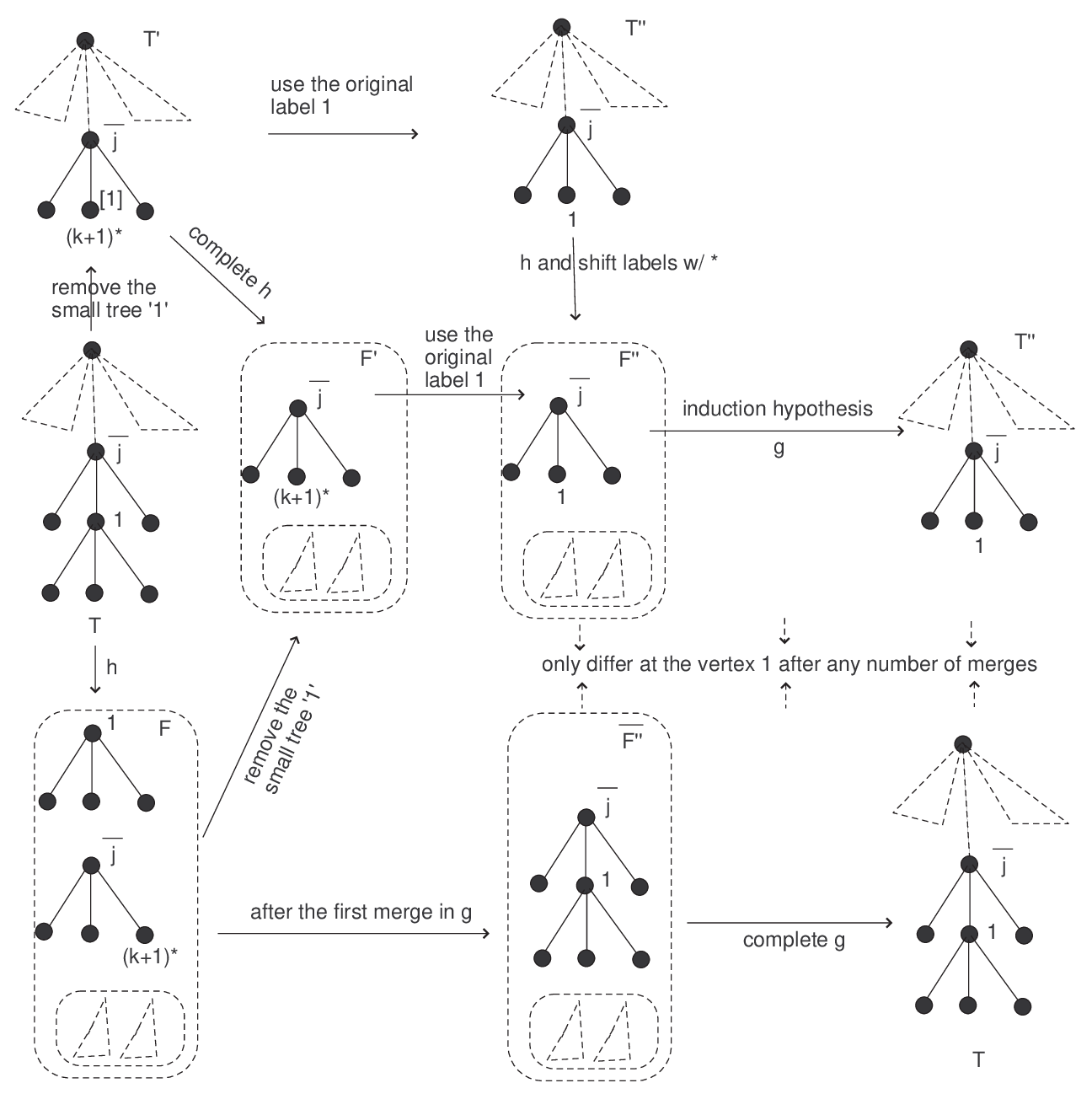}
		\caption{$g\circ h =id$.}
\label{fig:show-bij}
	\end{figure}

Based on Lemma~\ref{lem:loop-length}, Proposition~\ref{prop:eblock} and Theorem~\ref{thm:bij2}, we obtain the following theorem concerning enumeration of RNA secondary structures by the joint size distribution of helices and loops.

\begin{theorem}\label{thm:distribution}
The number of RNA secondary structures with $b+1$
base pairs, $k>1$ isolated bases, $l_e$ partial stacks and $s$ helices such that
the size distribution of the helices is $1^{a_1}2^{a_2} \cdots (b+1)^{a_{b+1}}$
and the size distribution of the loops is $1^{c_1} 2^{c_2} \cdots (b+k)^{c_{b+k}}$,
where $\sum_i ia_i=b+1$, $\sum_i a_i =s$, $\sum_i i c_i= b+k$ and $l_o=\sum_{i>1} c_i$,
is given by
\begin{align*}
\frac{s!}{\prod_{i>0} a_i !} \frac{(l_o-1)!}{\prod_{i>1} c_i !}  {k-1\choose l_e-1} {s-1 \choose l_o-1}{l_o \choose l_e+l_o-s}\;.
 \end{align*}
\end{theorem}
\begin{proof}
Following from Lemma~\ref{lem:loop-length} and Proposition~\ref{prop:eblock}, the secondary structures correspond to plane trees with $k$ even-level vertices,
$b+1$ odd-level vertices, $s$ E-blocks, $l_e$ even-level internal vertices and $l_o$ odd-level internal vertices. Using the bijection in Theorem~\ref{thm:bij2},
it suffices to enumerate the corresponding forests of the labelled trees from the latter.
Such a forest can be successively constructed as follows.

{\bf Label the roots of the small O-trees.} Note that there are $l_o$ small O-trees and
the corresponding roots have unstarred labels from $O=[\overline{b+1}]$.
So there are ${b+1 \choose l_o}$ different choices. (At the beginning,
we just need to determine which labels are roots.)

{\bf Label the roots of the small E-trees with unstarred roots.} It is clear that there are $l_e$ small E-trees (corresponding to $l_e$ partial stacks) with unstarred roots
and we have ${k\choose l_e}$ ways to pick $l_e$ unstarred labels from $E=[k]$.

{\bf Label the roots of the small E-trees with starred roots.}
\begin{figure}[!htb]
		\centering
		\includegraphics[width=1.0\textwidth]{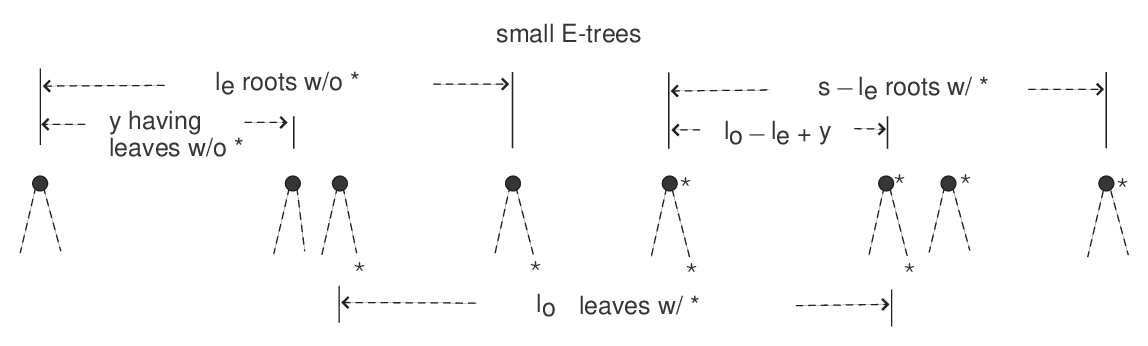}
		\caption{The distribution of the four types of small E-trees.}
\label{fig:small-e-trees}
	\end{figure}
Note that not every starred E-label can be a root of some small E-tree due to properties (c*)
and (d*) in Theorem~\ref{thm:bij2}
and thus we need first to exclude these starred E-labels.
Suppose there are $y$ small trees with unstarred roots
and unstarred leaves, see Figure~\ref{fig:small-e-trees}. There are ${l_e \choose y}$
different possibilities for assigning roots for these $y$ small trees from the $l_e$ unstarred labels picked in the last step. Accordingly, the smallest $y$ starred E-labels are leaves of some small O-trees by (c*) and thus cannot
be roots.
We know that there are in total $s$ small E-trees and $l_o$ of them have their rightmost leaves being starred.
Thus, there are $s-l_o-y$ small E-trees with starred roots but without starred leaves whose roots need to be determined.
Note that these starred E-labels can not be consecutive to guarantee that if $(k+m)^*$ is a root then
$(k+m+1)^* \leq (k+s-1)^*$ must be a leaf by (d*).
Here we distinguish two cases, see Figure~\ref{fig:non-consecutive}:
\begin{itemize}
\item The case of $(k+s-1)^*$ being a root. In this case, we have to pick $s-l_o-y-1$ non-consecutive labels from $(s-1)-y-2$
consecutive labels (i.e.~those unused and smaller than $(k+s-2)^*$), which can be done in
$$
{[(s-1)-y-2]-[s-l_o-y-1]+1 \choose s-l_o-y-1}={l_o-1 \choose s-l_o-y-1}
$$
different ways. In this case, the $s-l_o-y-1$ labels immediately following these non-consecutive labels
have to be assigned to some small O-trees as leaves.
Now there are $(s-1)-y-2[s-l_o-y]+1=y+2l_o-s$ unused starred labels,
and there are $l_o-(l_e-y)$ undetermined roots of small E-trees with starred roots.
Thus, we have ${y+2l_o-s \choose l_o-(l_e-y)}$ different choices for these roots and any remaining labels are assigned to leaves of small O-trees.
\begin{figure}[!htb]
		\centering
		\includegraphics[width=0.9\textwidth]{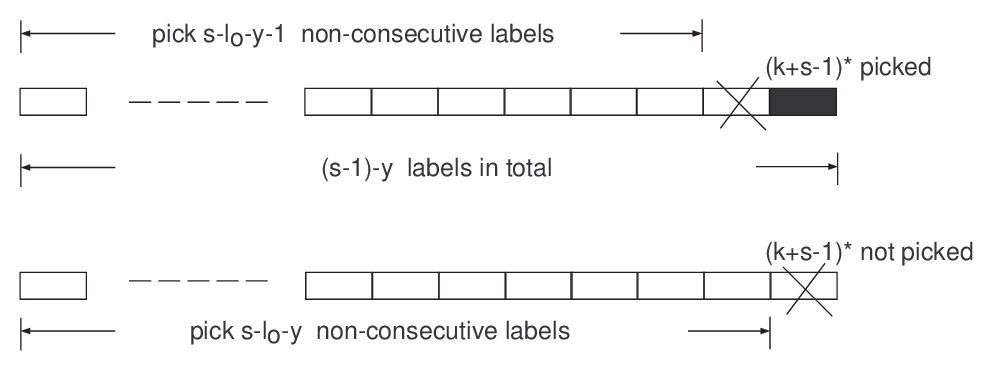}
		\caption{The two cases when picking non-consecutive labels.}
\label{fig:non-consecutive}
	\end{figure}
 \item The case of $(k+s-1)^*$ not being a root. In this case, we have to pick $s-l_o-y$ non-consecutive labels from $(s-1)-y-1$
consecutive labels (i.e.~those unused and smaller than $(k+s-1)^*$), which can be done in
$$
{[(s-1)-y-1]-[s-l_o-y]+1 \choose s-l_o-y}={l_o-1 \choose s-l_o-y}
$$
different ways. In this case, the $s-l_o-y$ labels immediately and respectively following these non-consecutive labels
have to be assigned to some small O-trees as leaves.
Now there are $(s-1)-y-2[s-l_o-y]=y+2l_o-s-1$ unused starred labels,
and there are $l_o-(l_e-y)$ undetermined roots of small E-trees with starred roots.
Thus, we have ${y+2l_o-s-1 \choose l_o-(l_e-y)}$ different choices for these roots and any remaining labels are assigned to leaves of small O-trees.
\end{itemize}

{\bf Label the leaves of the small O-trees.}
We first fix a linear order of the roots for the small O-trees, say in increasing order.
In order to guarantee the outdegree distribution of odd-level internal vertices, the size distribution of these (ordered)
small O-trees is a sequence of the type $1^{c_2} 2^{c_3} \cdots (b+k-1)^{c_{b+k}}$.
Obviously, there are $\frac{l_o!}{c_2! \cdots c_{b+k}!}$ distinct such sequences.
For each such a sequence, the leaves of these small O-trees are
a permutation of the $k-1$ unused labels from the set $[k] \bigcup \{ (k+1)^{*}, \ldots, (k+s-1)^*\}$,
that contributes the factor $(k-1)!$.

{\bf Label the leaves of the small E-trees.}
As for the leaves of the small $E$-trees, we fix a linear order of the roots for small E-trees.
Note that we have already determined which small E-trees (in terms of the roots) have starred O-labels (as their rightmost leaves).
Thus,
for each fixed sequence of type $1^{a_1} \cdots (b+1)^{a_{b+1}}$, we have to arrange the $l_o$ starred O-labels in $l_o!$ ways and arrange the $b+1-l_o$ unused labels from the set $[\overline{b+1}]$
in $(b+1-l_o)!$ different ways.

As a result, the number of such forests equals
\begin{align*}
& {b+1 \choose l_o}  {k\choose l_e} \sum_{y}{l_e \choose y}\frac{l_o!}{c_2! \cdots c_{b+k}!} (k-1)! l_o!(b+1-l_o)! \frac{s!}{a_1! \cdots a_{b+1}!}\\
 &\quad \times \left[{l_o-1 \choose s-l_o-y-1} {y+2l_o-s \choose l_o+y-l_e} + {l_o-1 \choose s-l_o-y} {y+2l_o-s-1 \choose l_o+y-l_e}\right]\\
 =& {b+1 \choose l_o}  {k\choose l_e} \sum_{y} {l_e \choose y}\frac{l_o! (k-1)! l_o!}{c_2! \cdots c_{b+k}!}  \frac{s! (b+1-l_o)!}{a_1! \cdots a_{b+1}!}
 {l_e-y \choose l_e+l_o-s}{l_o-1 \choose l_e-y-1},
\end{align*}
where the summation is over all possible values for $y$.
Since $k>1$, there is at least one odd-level internal vertex which is a child of the root of the tree,
implying at least one small E-tree with an unstarred root and a starred O-label.
Thus, we first require $y\leq l_e-1$. It is also possible that there is no young even-level internal vertices at all, thus $y\geq 0$. On the other hand, we know there are $l_e-y$ small E-trees with unstarred roots and starred leaves, since there are at most $l_o$ starred leaves, we have
$1\leq l_e-y \leq l_o$. Furthermore, ${l_o-1 \choose l_e-y-1}=0$ if $0\leq l_o-1< l_e-y-1$.
In view of this we have $0\leq y \leq l_e-1$.
It has been shown in Sprugnoli~\cite[eq.~$(6.14)$]{sprugnoli} that
$$
\sum_{k=0}^n {x \choose k}{y \choose n-k} {k \choose j}={x\choose j}{y+x-j \choose n-j}.
$$
Accordingly, we obtain
\begin{align*}
&\sum_{y=0}^{l_e-1} {l_e \choose y} {l_e-y \choose l_e+l_o-s}{l_o-1 \choose l_e-y-1}\\
=&\sum_{y=0}^{l_e-1} \frac{l_e}{l_e-y}{l_e-1 \choose y} \frac{l_e-y}{l_e+l_o-s}{l_e-y-1 \choose l_e+l_o-s-1}{l_o-1 \choose l_e-y-1}\\
=& \frac{l_e}{l_e+l_o-s} {l_o-1 \choose l_e+l_o-s-1} {l_e-1+l_o-1-(l_e+l_o-s-1) \choose l_e-1-(l_e+l_o-s-1)},
\end{align*}
that allows us to get a closed form for the total number of forests.
Dividing the total number of forests by $(b+1)! k!$, we obtain the desired number of RNA secondary structures and the proof is complete.
\end{proof}

Removing the restriction on the number of partial stacks, we obtain
\begin{corollary}
The number of RNA secondary structures with $b+1$
base pairs, $k>1$ isolated bases, $s$ helices such that
the size distribution of the helices is $1^{a_1}2^{a_2} \cdots (b+1)^{a_{b+1}}$
and the size distribution of the loops is $1^{c_1} 2^{c_2} \cdots (b+k)^{c_{b+k}}$,
where $\sum_i ia_i=b+1$, $\sum_i a_i =s$, $\sum_i i c_i= b+k$ and $l_o=\sum_{i>1} c_i$,
is given by
\begin{align*}
\frac{s!}{\prod_{i>0} a_i !} \frac{(l_o-1)!}{\prod_{i>1} c_i !}  {k-1+l_o\choose s-1} {s-1 \choose l_o-1} \; .
 \end{align*}
\end{corollary}
\begin{proof}
We just need to sum over all possible $1\leq l_e \leq s$:
\begin{align*}
&\sum_{l_e=1}^s \frac{s!}{\prod_{i>0} a_i !} \frac{(l_o-1)!}{\prod_{i>1} c_i !}  {k-1\choose l_e-1} {s-1 \choose l_o-1}{l_o \choose l_e+l_o-s}\\
=& \frac{s!}{\prod_{i>0} a_i !} \frac{(l_o-1)!}{\prod_{i>1} c_i !}  {k-1+l_o\choose s-1} {s-1 \choose l_o-1} \; ,
\end{align*}
completing the proof.
\end{proof}

\subsection{Exact enumeration by helices}
Now we enumerate RNA secondary structures simply by the number of helices alone.
We augment the formula considering the minimum helix size $\sigma$.
\begin{theorem}\label{thm:stack}
The number of RNA secondary structures with $b+1$
base pairs, $k>1$ isolated bases and $s$ helices such that
the size distribution of the helices is $1^{a_1}2^{a_2} \cdots (b+1)^{a_{b+1}}$ is
\begin{align*}
\frac{s!}{a_1! \cdots a_{b+1}!}\frac{1}{k-1}  \sum_{l_o=1}^s  {k-1 \choose l_o} {s-1 \choose l_o-1} {k-1+l_o \choose s-1} \; .
\end{align*}
Furthermore, the number of RNA secondary structures with $b+1$
base pairs, $k>1$ isolated bases and $s$ helices such that
each helix contains at least $\sigma$ base pairs
is given by
\begin{align*}
 {b-(\sigma-1)s \choose s-1} \frac{1}{k-1}\sum_{l_o=1}^s  {k-1 \choose l_o} {s-1 \choose l_o-1} {k-1+l_o \choose s-1} \; .
\end{align*}

\end{theorem}

\begin{proof}
In view of Theorem~\ref{thm:distribution}, we notice that, as long as
the number $s$ of E-blocks, the number $l_e$ of even-level internal vertices and the number $l_o$ of odd-level internal vertices are determined,
the enumeration of RNA secondary structures reduces to counting
the number of certain integer compositions, as each integer composition is associated with the same number
$$
{b+1\choose l_o} (k-1)!l_o!(b+1-l_o)! {k\choose l_e} \frac{l_e}{l_e+l_o-s}{s-1 \choose s-l_o}{l_o-1 \choose l_e+l_o-s-1}
$$
of labelled plane trees.
For the former part of the theorem, it reduces to counting
the number of integer compositions of $k-1$ into $l_o$ parts and $b+1$ into $s$ parts of type $1^{a_1}2^{a_2} \cdots (b+1)^{a_{b+1}}$ which are ${k-2 \choose l_o-1}$ and $\frac{s!}{a_1! \cdots a_{b+1}!}$, respectively.
Dividing $(b+1)!k!$ and summing over all possible $l_e$ and $l_o$, we have the desired number for the former part to be
\begin{align*}
& \frac{s!}{a_1! \cdots a_{b+1}!}\frac{1}{k}\sum_{l_e=1}^s {k\choose l_e} \sum_{l_o=1}^s {k-2 \choose l_o-1} \frac{l_e}{l_e+l_o-s}{s-1 \choose s-l_o}{l_o-1 \choose l_e+l_o-s-1}\\
= & \frac{s!}{a_1! \cdots a_{b+1}!} \frac{1}{k-1}\sum_{l_o=1}^s  {k-1 \choose l_o} {s-1 \choose l_o-1} {k-1+l_o \choose s-1}.
\end{align*}
For the latter part of the theorem,
since we require each helix to have size at least $\sigma$, we actually
need to consider the number of integer compositions of $b+1$ into $s$
parts such that each part has size at least $\sigma$ which is given by
${b-(\sigma-1)s \choose s-1}$. Each such an integer composition can be obtained as follows: allocate $\sigma-1$ to
each part first and then combine an arbitrary integer composition of $(b+1)-s(\sigma-1)$
into $s$ parts.
Dividing $(b+1)!k!$ and summing over all possible $l_e$ and $l_o$, we have the desired number to be
\begin{align*}
{b-(\sigma-1)s \choose s-1} \frac{1}{k-1}\sum_{l_o=1}^s  {k-1 \choose l_o} {s-1 \choose l_o-1} {k-1+l_o \choose s-1},
\end{align*}
completing the proof.
\end{proof}

\begin{example}
Some values obtained from Theorem~\ref{thm:stack} are shown in Table~\ref{table1} and Table~\ref{table2}.
For instance, for $b=2,~k=3,~\sigma=1$, the number of these RNA secondary structures with $s=3$ helices given by Theorem~\ref{thm:stack} is $9$,
and these secondary structures are presented in Figure~\ref{fig2}.

\begin{table}[h!]

\caption{The number of secondary structures for $\sigma=1$.}
\label{table1}
\begin{center}
\begin{tabular}{|c|c|c|c|c|c|c|c|c|c|}
\hline
& $b=2$ & $b=2$ & $b=2$ & $b=3$ & $b=3$ & $b=3$ & $b=4$ & $b=4$ & $b=4$\\
& $k=3$ & $k=4$ & $k=5$ & $k=3$ & $k=4$ & $k=5$ & $k=3$ & $k=4$ & $k=5$\\\hline
$s=1$ & $1$ & $1$ & $1$ & $1$ & $1$ & $1$ & $1$ & $1$ & $1$\\ \hline
$s=2$ & $10$ & $18$ & $28$ & $15$ & $27$ & $42$ & $20$ & $36$ & $56$ \\\hline
$s=3$ & $9$ & $31$ & $76$ & $27$ & $93$ & $228$ & $54$ & $186$ & $456$ \\\hline
$s=4$ & $0$ & $0$ & $0$ & $7$ & $54$ & $219$ & $28$ & $216$ & $876$ \\\hline
$s=5$ & $0$ & $0$ & $0$ & $0$ & $0$ & $0$ & $2$ & $51$ & $375$ \\\hline
\end{tabular}
\end{center}
\end{table}
\begin{table}[h!]

\caption{The number of secondary structures for $\sigma=2$.}
\label{table2}
\begin{center}
\begin{tabular}{|c|c|c|c|c|c|c|c|c|c|}
\hline
& $b=2$ & $b=2$ & $b=2$ & $b=3$ & $b=3$ & $b=3$ & $b=4$ & $b=4$ & $b=4$\\
& $k=3$ & $k=4$ & $k=5$ & $k=3$ & $k=4$ & $k=5$ & $k=3$ & $k=4$ & $k=5$\\\hline
$s=1$ & $1$ & $1$ & $1$ & $1$ & $1$ & $1$ & $1$ & $1$ & $1$\\ \hline
$s=2$ & $0$ & $0$ & $0$ & $5$ & $9$ & $14$ & $10$ & $18$ & $28$ \\\hline
\end{tabular}
\end{center}
\end{table}

	\begin{figure}[!htb]
		\centering
		\includegraphics[width=0.9\textwidth]{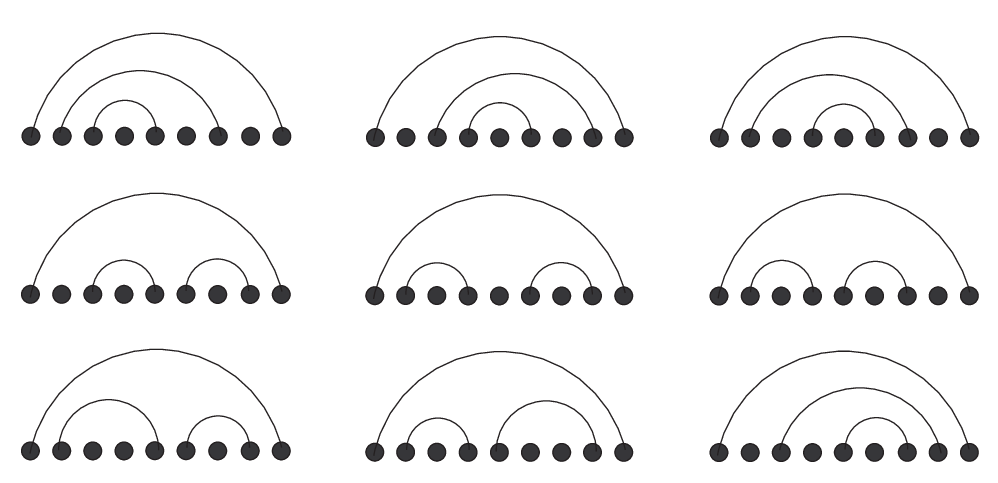}
		\caption{Nine RNA secondary structures for $(b,k,s)=(2,3,3)$.}
\label{fig2}
	\end{figure}

\end{example}

In particular, Theorem~\ref{thm:stack} allows us to compute the exact average number of helices and average size of helices. The asymptotics of these quantities were first obtained by Hofacker, Schuster and Stadler (1998).
For instance, it was shown that for $\sigma=1$, the average number of helices over secondary structures of length $n$ approaches $\frac{(1-\alpha)^2 (1+ \alpha)}{3-2\alpha} n$;
the probability for a helix having a size $l$ approaches $\frac{1-\alpha^2}{\alpha^2} \alpha^{2 l}$ for some $\alpha$.

As a consequence of Theorem~\ref{thm:stack}, we have

\begin{corollary}
Let $R$ be an RNA secondary structure chosen uniformly randomly from the set of RNA secondary structures with $b+1$
base pairs, $k$ isolated bases and $s$ helices such that
each has size at least $\sigma$.
Then, the probability of $R$ to have a helix size distribution $1^{a_1}2^{a_2} \cdots (b+1)^{a_{b+1}}$
is
$$
\frac{s! (s-1)! (b-\sigma s+1)!}{a_1!\cdots a_{b+1}! (b-\sigma s+s)!}
$$
where $a_i=0$ for $i<\sigma$.
\end{corollary}

\section{Conclusion}
In this paper, we obtained exact formulas counting RNA secondary structures
with a given number of helices as well as a given joint size distribution of helices and loops. Our approach was combinatorial, carefully analyzing the new bijection
between RNA secondary structures and plane trees discovered by the first author and a variation of Chen's bijective approach of counting trees by forests
of simple trees. While there are lots of results on the combinatorics of RNA
structures in the literature, only a few of them provide exact enumeration and
the rest provide asymptotic results. Those asymptotic results usually rely on
certain singularity analysis. However, singularity analysis of multivariate functions is not easy, which makes it hard to obtain enumerative results involving
multiple parameters, such as the joint distribution of helices and loops, and
thus restricts finer analyses of the space of RNA structures. In this regard,
our approach has advantages and may be further explored to enumerate RNA
structures filtered by other parameters and their combinations. We conclude
the paper with the following additional comments.
Partial stacks serve as an important intermediate parameter in order to
study other parameters such as helices and hairpin loops, when using Chen's
bijection. The distribution of partial stacks represents the distribution of left-end clusters of base pairs in RNA secondary structures. We may analogously
define ``partial stacks" associated with right-end clusters. However, the biological significance of partial stacks remains unclear at the moment, which may
be worthy of future investigation.

As one anonymous referee pointed out, it might be useful to consider the
requirement of the minimal base pair span $\theta$ (i.e.~if $(i,j)$ is
a base pair, then $|i-j|> \theta$), which is usually set to $3$.
The requirement of the minimal base
pair span to be $\theta$ is essentially the same as requiring the minimum length of
a hairpin loop to be $\theta$.
Based on the new bijection between RNA secondary
structures and plane trees, with extra effort, we can handle this requirement
as well. In fact, we have another separate paper which is particularly concerned
with hairpin loops, helices and bulges~\cite{chxx}.

\section*{Acknowledgments}
We would like to thank the anonymous referees for the valuable comments and
suggestions which improved the presentation of the paper.

\end{document}